\documentclass[fleqn]{article}
\usepackage[utf8]{inputenc}
\usepackage{
amsthm,amssymb,xcolor,hyperref,graphicx}
\usepackage[leqno]{amsmath}
\usepackage{tikz}
\makeatletter
\newcommand{\leqnomode}{\tagsleft@true}
\newcommand{\reqnomode}{\tagsleft@false}
\makeatother

\theoremstyle{plain}
\newtheorem{theorem}{Theorem}[section]

\newtheorem{observation}[theorem]{Observation}

\newtheorem{conjecture}[theorem]{Conjecture}
\newtheorem{remark}[theorem]{Remark}
\usepackage{latexsym}
\usepackage{amsfonts}
\usepackage[leqno]{amsmath}
\usepackage{tikz}
\usepackage{float}
\usepackage{lmodern}

\def\cl{\mathit{cl}}
\def\lk{\mathit{lk}}

\usepackage{marvosym}

\begin{document}
\title{Stable sets in flag spheres}

%% define your authors in the usual way
%% use \addressmark{1}, \addressmark{2} etc for the institutions, and use \thanks{} for contact details

%\author{Maria Chudnovsky \thanks{\href{mailto:mchudnov@math.princeton.edu}{mchudnov@math.princeton.edu}. Maria Chudnovsky was partially supported by Grant DMS-2120644 and by ISF grant 2480/20.}\addressmark{1}, \and Eran Nevo \thanks{\href{mailto:nevo@math.huji.ac.il}{nevo@math.huji.ac.il}. Eran Nevo was partially supported by the Israel Science Foundation grant ISF-2480/20 and by ISF-BSF joint grant 2016288.}\addressmark{2}}
%
%%% then use \addressmark to match authors to institutions here
%\address{\addressmark{1}Princeton University, Princeton, NJ 08544, USA \\ \addressmark{2}Einstein Institute of Mathematics, Hebrew University, Jerusalem, Israel}
\author{Maria Chudnovsky\thanks{Partially supported by  NSF-EPSRC Grant DMS-2120644 and by ISF grant 2480/20.}\\
Princeton University, Princeton, NJ 08544, USA
\\
\\
Eran Nevo\thanks{Partially supported by the Israel Science Foundation grant ISF-2480/20 and by ISF-BSF joint grant 2016288.}\\
Hebrew University, Jerusalem, Israel}

%% put the date of submission here
%\received{\today}
\date {
%July 16, 2021; revised
\today}

%\author{Maria Chudnovsky\thanks{\href{mailto:hello@world.c}{hello@world.c}. Partially supported by  NSF-EPSRC Grant DMS-2120644 and by ISF grant 2480/20.}\addressmark{1}, \and Eran Nevo\thanks{\href{mailto:hello@world.c}{hello@world.c}. Partially supported by the Israel Science Foundation grant ISF-2480/20 and by ISF-BSF joint grant 2016288.}\addressmark{2}}

%% then use \addressmark to match authors to institutions here
%\address{\addressmark{1}Princeton University, Princeton, NJ 08544, USA \\ \addressmark{2}Hebrew University, Jerusalem, Israel}

%% put the date of submission here
%\received{\today}

%\author{Maria Chudnovsky\thanks{Partially supported by  NSF-EPSRC Grant DMS-2120644 and by ISF grant 2480/20.}\\
%	Princeton University, Princeton, NJ 08544, USA
%	\\
%	\\
%	Eran Nevo\thanks{Partially supported by the Israel Science Foundation grant ISF-2480/20 and by ISF-BSF joint grant 2016288.}\\
%	Hebrew University, Jerusalem, Israel}
%
%\date {July 16, 2021; revised \today}
\maketitle
\begin{abstract}
%\abstract{
We provide lower and upper bounds on the minimum size of a maximum stable set over  graphs of flag spheres, as a function of the dimension of the sphere and the number of vertices.
%As an application of the lower bound,
Further, we use stable sets
to obtain an improved Lower Bound Theorem for the face numbers of flag spheres.
%}
\end{abstract}

\section{Introduction}
Given a graph $G$, a set $X \subseteq V(G)$ is {\em stable} (or
{\em independent}) if no edge of $G$ has both ends in $X$.
We denote by $\alpha(G)$ the size of a largest stable set in $G$;
a stable set of size $\alpha(G)$ is called a {\em maximum stable
  set} of $G$.
Stable sets are a basic concept in graph theory, but it is in general
very difficult to understand what the structure of maximum stable sets is
(this is related to the fact that the problem of computing $\alpha(G)$ is NP-complete). In this paper we study maximum stable sets in graphs
whose clique complex is topologically a sphere of fixed dimension
(these are called {\em graphs of flag spheres}). These graphs possess
a beautiful recursive structure, since the neighborhood of every vertex is
a graph of the same type but of lower dimension. They are also
of great interest in topological combinatorics and beyond, e.g. in the study of manifolds with nonpositive sectional curvature, via the Charney-Davis conjecture~\cite{Charney-Davis, Gal}.

Our main objective is the following natural invariant: the minimum size over maximum stable sets in $n$-vertex graphs of flag $(d-1)$-dimensional spheres, namely
$$
\alpha(d,n)= \min(\alpha(G):\ |V(G)|=n,\ \text{$\cl(G)$ triangulates the  $(d-1)$-dimensional sphere})
.$$
(Here $\cl(G)$ is the complex of cliques of $G$.)
For fixed $d$ we are interested in the growth of $\alpha(d,n)$ as $n\rightarrow \infty$.

\begin{conjecture}\label{conj:alpha}
For every $d\ge 2$ and $n\ge 2d$,   $\alpha(d,n)=\lceil\frac{n+d-3}{2(d-1)}\rceil$.
%For even $d\ge 2$,  $\alpha(d,n)=
%\lfloor\frac{n}{2(d-1)}\rfloor$.

%For odd $d\geq 3$,
%$\alpha(d,n)=
%\lceil\frac{n}{2(d-1)}\rceil$.
\end{conjecture}
This conjecture holds for $d=2$ (easy) and $d=3$ (see Theorem~\ref{alpha2dim}, using the 4-color theorem (4CT) for the lower bound).
For $d=4$ we prove that the conjectured upper bound holds.
% for
%infinitely many values of $n$.
%(\eran{IMPROVE TO ALL $n$ by edge sudivisions. USE the subdivisions at $a$ and $b$ in the CONSTRUCTION.})
For general $d\ge 4$ we show:
\begin{theorem}\label{tmh:alpha_bounds}
Let $d\ge 4$ and $n\ge 2d$.
Then
%for all $n$ large enough,
$$ \frac{1}{4} n^{\frac{1}{d-2}}\le
\alpha(d,n)\le
%\frac{n}{6\lfloor d/4\rfloor}.$$
\Bigl\lceil\frac{\lceil\frac{n}{\lfloor d/4\rfloor}\rceil +1}{6}\Bigr\rceil.$$
 \end{theorem}
The lower bound slightly improves on the Ramsey bound ($\Omega(n^{\frac{1}{d}})$)
by using the 4CT within the base case $d=4$. The upper bound, which is roughly $\frac{2n}{3d}$ for large $d$,
%(and exactly so for $d$ divisible by $4$),
is obtained by taking the join of copies
of the best flag $3$-spheres constructed in Theorem~\ref{flag3dim} for the upper bound, and taking up to $3$ extra suspensions to reach dimension $d-1$.
Indeed, a maximum stable set in the join is a maximum stable set in a component of the join -- now, ignoring rounding, such component is a 3-sphere on a $4/d$ fraction of the $n$ vertices, and a $1/6$ fraction of its vertices form a maximum stable set.

%A construction realizing the conjectured bound will follow provided we construct a small...

Our second result is an improved lower bound theorem on the number of edges for the class of flag spheres; the proof relies on the existence of a large stable
set in such graphs.
%The basic idea appears in the proof of~\cite[Lem.4.2]{Nevo-Missing}, and
Deducing from this bound lower bounds on the number of higher dimensional $k$-faces appeared in the proof of~\cite[Prop.3.2]{Nevo-Missing}, following the MPW-reduction.
\begin{theorem}\label{thm:flag-LBT-half}
(i) Fix $\delta>0$. There exists $d(\delta)$ such that for all $d\ge d(\delta)$ and $n$ large enough, each $n$-vertex flag $(d-1)$-sphere %$\Delta$
has at least
%$(d+\frac{1}{2}-\delta)n$
$(d+\frac{1-\delta}{2d+1})n$
edges.

(ii) For all $d\geq 6$, and $n$ large enough, each $n$-vertex flag $(d-1)$-sphere %$\Delta$
has at least
%$(d+0.294)n$
$(d+\frac{0.987}{2d+1})n$
edges.
\end{theorem}
Note that the Lower Bound Theorem for simplicial spheres~\cite{Barnette:LBT-73, Kalai:LBT} guarantees in (i) for simplicial spheres
%$f_1(\Delta)\geq (d-\delta)n$,
at least $(d-\delta)n$ edges
and Gal's conjecture~\cite{Gal},
which, if true, is tight,
 would imply
%$f_1(\Delta)\geq (2d-3-\delta)n$
at least $(2d-3-\delta)n$ edges
(it does hold for $d\leq 5$).
For $d\ge 6$ the lower bound in Theorem~\ref{thm:flag-LBT-half}(ii) appears to be new.
If Conjecture~\ref{conj:alpha} holds then this lower bound would further improve to at least
%$(d+0.4-\epsilon)n$
$(d+\frac{1}{2d-2})n$
edges,
for all $d\ge 6$,
for
%any fixed $\epsilon >0$ and
large enough $n$.

\textbf{Outline}:
In Section~\ref{sec:construction} we construct low dimensional flag spheres whose maximum independent sets are small, proving Conjecture~\ref{conj:alpha} for $d=3$ and the upper bound there for $d=4$, and deducing both bounds in Theorem~\ref{tmh:alpha_bounds}.
In Section~\ref{sec:f1Turan} we prove Theorem~\ref{thm:flag-LBT-half} by combining stable sets with framework rigidity.
In Section~\ref{sec:MaxStable} we give some results and conjectures regarding the corresponding invariant for the other extreme:
$$
\alpha_M(d,n)= \max(\alpha(G):\ |V(G)|=n,\  \text{$\cl(G)$ triangulates the $(d-1)$-dimensional sphere})
.$$

\section{The construction}\label{sec:construction}
We construct graphs, denoted $W_{d,k}$. First we analyze their $\alpha$, and next we analyze their clique complex. Figure~\ref{Fig-W33}(middle) illustrates $W_{3,3}$.

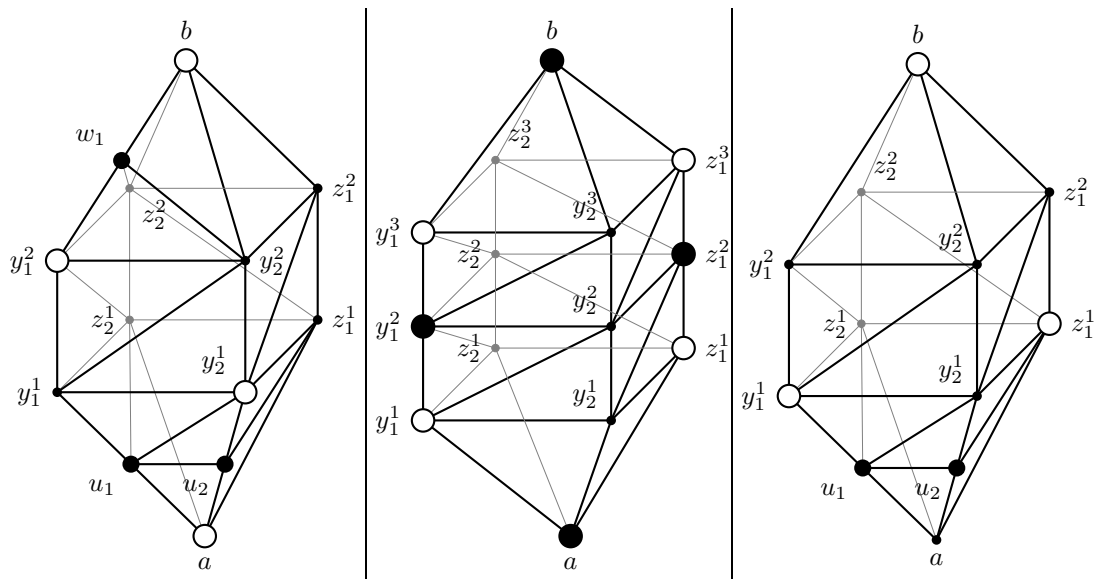
\begin{figure}[H]
\begin{minipage}{0.4\textwidth}
\begin{center}
	\begin{tikzpicture}%
	[
	scale=2.5,
	back/.style={gray, very thin},
	edge/.style={black, thick},
	vertex/.style={inner sep=1pt,circle,draw=black,fill=black,thick,anchor=base},
	wvertex/.style={inner sep=3pt,circle,draw=black,fill=white,thick,anchor=base},
	bvertex/.style={inner sep=3pt,circle,draw=black,fill=black,thick,anchor=base},
	mvertex/.style={inner sep=2pt,circle,draw=black,fill=black,thick,anchor=base},
	]
	
	%% vertex locations
	\coordinate (b) at (-0.43, 1.3, -0.3);
	
	\coordinate (u3) at (-0.715, 0.825, -0.15);
	
	\coordinate (m1) at (-1, 0.35, 0);
	\coordinate (m2) at (0, 0.35, 0);
	\coordinate (m3) at (0, 0.35, -1);
	\coordinate (m4) at (-1, 0.35, -1);
	
	\coordinate (b1) at (-1, -0.35, 0);
	\coordinate (b2) at (0, -0.35, 0);
	\coordinate (b3) at (0, -0.35, -1);
	\coordinate (b4) at (-1, -0.35, -1);
	
	\coordinate (u1) at (-0.55, -0.675, 0.15);
	\coordinate (u2) at (-0.05, -0.675, 0.15);
	
	\coordinate (a) at (-0.1, -1, 0.3);
	
	%% edges before vertices so they are "behind"
	
	%% edges from "b" to top layer
	\draw[edge] (b) -- (m1);
	\draw[edge] (b) -- (m2);
	\draw[edge] (b) -- (m3);
	\draw[edge,back] (b) -- (m4);
	
	%% all back edges
	\draw[edge,back] (m3) -- (m4);
	\draw[edge,back] (m4) -- (m1);
	\draw[edge,back] (m4) -- (b4);
	\draw[edge,back] (b3) -- (m4);
	\draw[edge,back] (b4) -- (m1);
	\draw[edge,back] (b3) -- (b4);
	\draw[edge,back] (b4) -- (b1);
	\draw[edge,back] (a) -- (b4);
	\draw[edge,back] (u1) -- (b4);
	
	%% edges in middle layer
	\draw[edge] (m1) -- (m2);
	\draw[edge] (m2) -- (m3);
	%	\draw[edge,back] (m3) -- (m4);
	%	\draw[edge,back] (m4) -- (m1);
	
	%% edges between middle and bottom layer
	\draw[edge] (m1) -- (b1);
	\draw[edge] (m2) -- (b2);
	\draw[edge] (m3) -- (b3);
	%	\draw[edge,back] (m4) -- (b4);
	
	\draw[edge] (b1) -- (m2);
	\draw[edge] (b2) -- (m3);
	%	\draw[edge,back] (b3) -- (m4);
	%	\draw[edge,back] (b4) -- (m1);
	
	%% edges in bottom layer
	\draw[edge] (b1) -- (b2);
	\draw[edge] (b2) -- (b3);
	%	\draw[edge,back] (b3) -- (b4);
	%	\draw[edge,back] (b4) -- (b1);
	
	%% edges from "a" to bottom layer
	\draw[edge] (a) -- (b1);
	\draw[edge] (a) -- (b2);
	\draw[edge] (a) -- (b3);
	%	\draw[edge,back] (a) -- (b4);
	
	%% edges involving u1, u2, u3
	\draw[edge] (u1) -- (u2);
	\draw[edge] (u1) -- (b2);
	\draw[edge] (u2) -- (b3);
	
	\draw[edge] (u3) -- (m2);
	\draw[edge,back] (u3) -- (m4);	
	
	%% draw vertices
	\node[wvertex,label=above:{$b$}] at (b)     {};
	
	\node[mvertex,label=above left:{$w_1$}] at (u3)     {};
	
	\node[wvertex,label=left:{$y_1^2$}] at (m1)     {};
	\node[vertex,label=right:{$y_2^2$}] at (m2)     {};
	\node[vertex,label=right:{$z_1^2$}] at (m3)     {};
	\node[vertex,label=below right:{$z_2^2$},back] at (m4)     {};
	
	\node[vertex,label=left:{$y_1^1$}] at (b1)     {};
	\node[wvertex,label=above left:{$y_2^1$}] at (b2)     {};
	\node[vertex,label=right:{$z_1^1$}] at (b3)     {};
	\node[vertex,label=left:{$z_2^1$},back] at (b4)     {};
	
	\node[mvertex,label=below left:{$u_1$}] at (u1)     {};
	\node[mvertex,label=below left:{$u_2$}] at (u2)     {};
	
	\node[wvertex,label=below:{$a$}] at (a)     {};
	
	\end{tikzpicture}
\end{center}
\end{minipage}\vline
\begin{minipage}{0.4\textwidth}
\begin{center}
	\begin{tikzpicture}%
	[
	scale=2.5,
	back/.style={gray, very thin},
	edge/.style={black, thick},
	vertex/.style={inner sep=1pt,circle,draw=black,fill=black,thick,anchor=base},
	wvertex/.style={inner sep=3pt,circle,draw=black,fill=white,thick,anchor=base},
	bvertex/.style={inner sep=3pt,circle,draw=black,fill=black,thick,anchor=base},
	]
	
	%% vertex locations
	\coordinate (b) at (-0.43, 1.3, -0.3);
	
	\coordinate (t1) at (-1, 0.5, 0);
	\coordinate (t2) at (0, 0.5, 0);
	\coordinate (t3) at (0, 0.5, -1);
	\coordinate (t4) at (-1, 0.5, -1);
	
	\coordinate (m1) at (-1, 0, 0);
	\coordinate (m2) at (0, 0, 0);
	\coordinate (m3) at (0, 0, -1);
	\coordinate (m4) at (-1, 0, -1);
	
	\coordinate (b1) at (-1, -0.5, 0);
	\coordinate (b2) at (0, -0.5, 0);
	\coordinate (b3) at (0, -0.5, -1);
	\coordinate (b4) at (-1, -0.5, -1);
	
	\coordinate (a) at (-0.1, -1, 0.3);
	
	%% edges before vertices so they are "behind"
	
	%% edges from "b" to top layer
	\draw[edge] (b) -- (t1);
	\draw[edge] (b) -- (t2);
	\draw[edge] (b) -- (t3);
	\draw[edge,back] (b) -- (t4);
	
	%% edges in top layer
	\draw[edge] (t1) -- (t2);
	\draw[edge] (t2) -- (t3);
	\draw[edge,back] (t3) -- (t4);
	\draw[edge,back] (t4) -- (t1);
	
	%% edges between top and middle layer
	\draw[edge] (t1) -- (m1);
	\draw[edge] (t2) -- (m2);
	\draw[edge] (t3) -- (m3);
	\draw[edge,back] (t4) -- (m4);
	
	\draw[edge] (m1) -- (t2);
	\draw[edge] (m2) -- (t3);
	\draw[edge,back] (m3) -- (t4);
	\draw[edge,back] (m4) -- (t1);

	%% edges in middle layer
	\draw[edge] (m1) -- (m2);
	\draw[edge] (m2) -- (m3);
	\draw[edge,back] (m3) -- (m4);
	\draw[edge,back] (m4) -- (m1);
	
	%% edges between middle and bottom layer
	\draw[edge] (m1) -- (b1);
	\draw[edge] (m2) -- (b2);
	\draw[edge] (m3) -- (b3);
	\draw[edge,back] (m4) -- (b4);

	\draw[edge] (b1) -- (m2);
	\draw[edge] (b2) -- (m3);
	\draw[edge,back] (b3) -- (m4);
	\draw[edge,back] (b4) -- (m1);
	
	%% edges in bottom layer
	\draw[edge] (b1) -- (b2);
	\draw[edge] (b2) -- (b3);
	\draw[edge,back] (b3) -- (b4);
	\draw[edge,back] (b4) -- (b1);
	
	%% edges from "a" to bottom layer
	\draw[edge] (a) -- (b1);
	\draw[edge] (a) -- (b2);
	\draw[edge] (a) -- (b3);
	\draw[edge,back] (a) -- (b4);
		
	%% draw vertices
	\node[bvertex,label=above:{$b$}] at (b)     {};
	
	\node[wvertex,label=left:{$y_1^3$}] at (t1)     {};
	\node[vertex,label=above left:{$y_2^3$}] at (t2)     {};
	\node[wvertex,label=right:{$z_1^3$}] at (t3)     {};
	\node[vertex,label=above right:{$z_2^3$},back] at (t4)     {};
	
	\node[bvertex,label=left:{$y_1^2$}] at (m1)     {};
	\node[vertex,label=above left:{$y_2^2$}] at (m2)     {};
	\node[bvertex,label=right:{$z_1^2$}] at (m3)     {};
	\node[vertex,label=left:{$z_2^2$},back] at (m4)     {};
	
	\node[wvertex,label=left:{$y_1^1$}] at (b1)     {};
	\node[vertex,label=above left:{$y_2^1$}] at (b2)     {};
	\node[wvertex,label=right:{$z_1^1$}] at (b3)     {};
	\node[vertex,label=left:{$z_2^1$},back] at (b4)     {};
	
	\node[bvertex,label=below:{$a$}] at (a)     {};
	
	\end{tikzpicture}
\end{center}
\end{minipage}\vline
\begin{minipage}{0.4\textwidth}
\begin{center}
	\begin{tikzpicture}%
	[
	scale=2.5,
	back/.style={gray, very thin},
	edge/.style={black, thick},
	vertex/.style={inner sep=1pt,circle,draw=black,fill=black,thick,anchor=base},
	wvertex/.style={inner sep=3pt,circle,draw=black,fill=white,thick,anchor=base},
	bvertex/.style={inner sep=3pt,circle,draw=black,fill=black,thick,anchor=base},
	mvertex/.style={inner sep=2pt,circle,draw=black,fill=black,thick,anchor=base},
	]
	
	%% vertex locations
	\coordinate (b) at (-0.43, 1.3, -0.3);
	
	\coordinate (m1) at (-1, 0.35, 0);
	\coordinate (m2) at (0, 0.35, 0);
	\coordinate (m3) at (0, 0.35, -1);
	\coordinate (m4) at (-1, 0.35, -1);
	
	\coordinate (b1) at (-1, -0.35, 0);
	\coordinate (b2) at (0, -0.35, 0);
	\coordinate (b3) at (0, -0.35, -1);
	\coordinate (b4) at (-1, -0.35, -1);
	
	\coordinate (u1) at (-0.55, -0.675, 0.15);
	\coordinate (u2) at (-0.05, -0.675, 0.15);
	
	\coordinate (a) at (-0.1, -1, 0.3);
	
	%% edges before vertices so they are "behind"
	
	%% edges from "b" to top layer
	\draw[edge] (b) -- (m1);
	\draw[edge] (b) -- (m2);
	\draw[edge] (b) -- (m3);
	\draw[edge,back] (b) -- (m4);
	
	%% all back edges
	\draw[edge,back] (m3) -- (m4);
	\draw[edge,back] (m4) -- (m1);
	\draw[edge,back] (m4) -- (b4);
	\draw[edge,back] (b3) -- (m4);
	\draw[edge,back] (b4) -- (m1);
	\draw[edge,back] (b3) -- (b4);
	\draw[edge,back] (b4) -- (b1);
	\draw[edge,back] (a) -- (b4);
	\draw[edge,back] (u1) -- (b4);
	
	%% edges in middle layer
	\draw[edge] (m1) -- (m2);
	\draw[edge] (m2) -- (m3);
	%	\draw[edge,back] (m3) -- (m4);
	%	\draw[edge,back] (m4) -- (m1);
	
	%% edges between middle and bottom layer
	\draw[edge] (m1) -- (b1);
	\draw[edge] (m2) -- (b2);
	\draw[edge] (m3) -- (b3);
	%	\draw[edge,back] (m4) -- (b4);
	
	\draw[edge] (b1) -- (m2);
	\draw[edge] (b2) -- (m3);
	%	\draw[edge,back] (b3) -- (m4);
	%	\draw[edge,back] (b4) -- (m1);
	
	%% edges in bottom layer
	\draw[edge] (b1) -- (b2);
	\draw[edge] (b2) -- (b3);
	%	\draw[edge,back] (b3) -- (b4);
	%	\draw[edge,back] (b4) -- (b1);
	
	%% edges from "a" to bottom layer
	\draw[edge] (a) -- (b1);
	\draw[edge] (a) -- (b2);
	\draw[edge] (a) -- (b3);
	%	\draw[edge,back] (a) -- (b4);
	
	%% edges involving u1, u2
	\draw[edge] (u1) -- (u2);
	\draw[edge] (u1) -- (b2);
	\draw[edge] (u2) -- (b3);

	%% draw vertices
	\node[wvertex,label=above:{$b$}] at (b)     {};
	
	\node[vertex,label=left:{$y_1^2$}] at (m1)     {};
	\node[vertex,label=above left:{$y_2^2$}] at (m2)     {};
	\node[vertex,label=right:{$z_1^2$}] at (m3)     {};
	\node[vertex,label=above right:{$z_2^2$},back] at (m4)     {};
	
	\node[wvertex,label=left:{$y_1^1$}] at (b1)     {};
	\node[vertex,label=above left:{$y_2^1$}] at (b2)     {};
	\node[wvertex,label=right:{$z_1^1$}] at (b3)     {};
	\node[vertex,label=left:{$z_2^1$},back] at (b4)     {};
	
	\node[mvertex,label=below left:{$u_1$}] at (u1)     {};
	\node[mvertex,label=below left:{$u_2$}] at (u2)     {};
	
	\node[vertex,label=below:{$a$}] at (a)     {};
	
	\end{tikzpicture}
\end{center}
\end{minipage}
%  \centering
%  \includegraphics[width=\textwidth]{Drawing-W33.pdf}
\caption{\textbf{Middle}: The graph $W_{3,3}$ is depicted. The bold black and bold white vertices indicate stable sets of size $\alpha(W_{3,3})=4$.
The shaded edges indicate
edges that are not
visible from a front view of the depicted realization of the flag $2$-sphere $\cl(W_{3,3})$ in $3$-space.
Similarly,
\textbf{Right}: the graph $X(3,2,2)$ is depicted. The bold white vertices indicate a stable set of size $\alpha(X(3,2,2))=3$;
\textbf{Left}: the graph $Y(3,2,1)$ is depicted. The bold white vertices indicate a stable set of size $\alpha(Y(3,2,1))=4$.
}\label{Fig-W33}
\end{figure}

Fix
%the dimension
an integer $d\geq 2$. For $k \geq 1$ let $W_{d,k}$ be the following
graph.
$V(W_{d,k})=\{a,b\} \cup X_1 \cup \ldots \cup  X_k$ where the sets
$X_1, \dots, X_k,\{a,b\}$ are pairwise disjoint and $|X_i|=2d-2$ for
every
$i \in \{1, \ldots, k\}$.
Denote $X_i=\{y_1 ^i,\ldots,y_{d-1} ^i, z_1 ^{i'},\ldots,z_{d-1} ^i\}$.
Next we list the edges of $W_{d,k}$.
\begin{itemize}
\item  $a$ is complete to $X_1$
and $b$ is complete to $X_k$ and there are no other edges incident
with
$a,b$.
\item  For every $i$, the induced graph $W_{d,k}[X_i]$ is the $1$-skeleton of the
$(d-1)$-dimensional crosspolytope, a.k.a. the graph of the octahedral $(d-2)$-sphere, with non-edges $y_1^iz_1^i, \ldots,
y_{d-1}^iz_{d-1}^i$.
\item $X_i$ is anticomplete to $X_j$ if $|i-j|>1$.
\item For  $i \in \{1, \dots, k-1\}$ and $s,t \in \{1, \dots, d-1\}$
    let us
  say that the pair $(y_s^iz_s^i,y_t^{i+1}z_t^{i+1})$ is   {\em
  positive} if
  $y_s^iy_t^{i+1}$ and $z_s^iz_t^{i+1}$ are edges, and
  $y_s^iz_t^{i+1}$ and
  $z_s^iy_t^{i+1}$ are non-edges, and {\em negative}  if
 $y_s^iy_t^{i+1}$ and $z_s^iz_t^{i+1}$ are non-edges, and
 $y_s^iz_t^{i+1}$ and
  $z_s^iy_t^{i+1}$ are edges. Then the pair
  $(y_s^iz_s^i,y_t^{i+1}z_t^{i+1})$
  is positive if $t \geq s$ and negative if $t<s$.
  \item All pairs of vertices of $W_{d,k}$ that are not mentioned
      above are non-edges.
\end{itemize}

Now we define certain edge subdivisions on $\cl(W_{d,k})$.
Consider a maximal simplex in the link of $a$ (resp. $b$) in $\cl(W_{d,k})$, say $y_1^1y_2^1y_3^1 \dots y_{d-1}^1$ (resp. $y_1^ky_2^ky_3^k \dots y_{d-1}^k$).
Given a simplicial complex $Z$ and an edge $xy$ of $Z$, we denote by
$Z(xy)$ the complex obtained from $Z$ by the stellar subdivision of $Z$ at $xy$ (also called \emph{edge subdivision}),
and by $v_{xy}$ the new vertex resulting from such a subdivision.
Make the following sequence of $2d-2$ edge subdivisions:

$X''(d,k,0):=\cl(W_{d,k})$,
and for $j \in \{1, \dots, d-1\}$, having defined $X''(d,k,j-1)$ and $u_{j-1}$
(for $j>1$),
let $X''(d,k,j):=X''(d,k,j-1)(ay_j^1)$ and  $u_j:=v_{ay_j^1}$.
Let $X(d,k,j)$ be the graph that is the $1$-skeleton of $X''(d,k,j)$ (thus
$X(d,k,0)=W_{d,k}$). For example, Figure~\ref{Fig-W33}(right) illustrates $X(3,2,2)$.

Next let $Y''(d,k,0):=X''(d,k,d-1)$,
and for $j \in \{1, \dots, d-1\}$, having defined $Y''(d,k,j-1)$ and $w_{j-1}$
(for $j>1$),
let $Y''(d,k,j):=Y''(d,k,j-1)(by_j^k)$ and $w_j:=v_{by_j^k}$.
Let $Y(d,k,j)$ be the graph that is the $1$-skeleton of $Y''(d,k,j)$.
For example, Figure~\ref{Fig-W33}(left) illustrates $Y(3,2,1)$.

\begin{theorem}\label{alpha}
  For every $d \geq 2, k \geq 1, d-1\geq j \geq 0$,\\
  $\alpha(X(d,k,j))=k+1= \frac{|V(X(d,k,j))|-2-j}{2d-2}+1$.\\
  For every $d \geq 3, k \geq 1, d-1\geq j \geq 1$,\\
  $\alpha(Y(d,k,j))=k+2= \frac{|V(Y(d,k,j))|-2+(d-1-j)}{2d-2}+1$.
\end{theorem}
\begin{proof}

  Let $G$ be one of the graphs $X(d,k,j)$ or $Y(d,k,j)$. Let
  $U$ be the set of vertices of the form $u_j$ in $G$, and
  let $W$ be the set of vertices of the form $w_j$ in $G$.
  Then $W \neq \emptyset$ only if $|U|=d-1$. Moreover $U \cup a$
  and $W \cup b$ are both cliques in $G$.
  Denote by $N_G(v)$ the neighbors of $v$ in $G$. Then, $X_1 \setminus N_G(a) \subseteq \{y_1^1, \dots, y_{d-1}^1\}$,
  and for every $j$ we have that
  $X_1 \setminus N_G(u_j) = \{y_1^1, \dots, y_{j-1}^1, z_{j}^1\}$.
  In particular, $\alpha(G[X_1 \setminus N_G(v)]) \leq 1$ for
  every $v \in U \cup \{a\}$.
  Similarly, $\alpha(G[X_k \setminus N_G(v)]) \leq 1$ for
  every $v \in W \cup \{b\}$.

  Let $S$ be a stable set of $G$. First we prove an upper bound on $|S|$.
  Clearly for every $i$ we have that $\alpha(G[X_i])=2$. Moreover every
  vertex
  of $X_{i+1}$ has a neighbor in every non-edge of $G[X_i]$, and every
  vertex of
  $X_i$ has a neighbor in every non-edge of $G[X_{i+1}]$.
Consequently,   $|S \cap (X_i \cup X_{i+1})| \leq 2$.

  Hence $|S \setminus (U \cup W \cup \{a,b\})| \leq k+1$.
  Suppose $|S \setminus (U \cup W \cup \{a,b\})| = k+1$. Then $k$ is odd, and
  $|S \cap X_1|=|S \cap X_k|=2$.
  It follows that  $S \cap (U \cup W \cup \{a,b\})=\emptyset$ and
  $|S|= k+1$.

Next suppose that   $|S \setminus (U \cup W \cup \{a,b\})| = k$.
Since $U\cup\{a\}$ and $W\cup\{b\}$ are both cliques, it follows that $|S| \leq k+2$,
and so we may assume that $G=X(d,k,j)$ for some $j$ (for otherwise $G=Y(d,k,j)$ and the
upper bound on $\alpha(G)$ holds).
%%%%%%%%%%%%%%
%With some extra work, of similar flavor, one shows that
%$|S| \leq k+1$ in this case, since $W=\emptyset$ in this case.
%%%%%%%%%%%%%%%%

In particular $W=\emptyset$ and $b$ is
adjacent to every vertex of $X_k$.
Since  $|S \setminus (U \cup W \cup \{a,b\})| = k$, it follows that
$|S \cap (X_1 \cup X_k)|=2$.
If $|S \cap X_k| \neq \emptyset$, then $b \not \in S$,
and, since $U \cup \{a\}$ is a clique, we have that $|S| \leq k+1$.
Thus we may assume that $S \cap X_k = \emptyset$, and
so $|S \cap X_1|=2$. Since $\alpha(G[X_1 \setminus N(v)]) \leq 1$ for
every $v \in U \cup \{a\}$, we deduce that $S \cap (U \cup \{a\})=\emptyset$,
and so $|S|=k$ if $b\notin S$ and $|S|=k+1$ if $b\in S$.

Clearly if $|S \setminus \{a,b\}| < k$ then,
since $U \cup \{a\}$ and $W \cup \{b\}$ are both cliques,
we have that $|S| \leq k+1$.
Thus in all cases the upper bound on $|S|$ holds.

Next we show that if $G=X(d,k,j)$ for some $j \geq 0$ then   $\alpha(G)=k+1$.
Let
$S'=\bigcup_{i \in {1, \dots, k};\  i \text{ odd}}\{y_1^i,z_1^i\}$.
If $k$ is odd let $S=S'$. If $k$ is even, let
$S=S' \cup \{b\}$. In both cases $|S|=k+1$.

Finally we show that  if $G=Y(d,k,j)$ for some $j \geq 1 $  then
$\alpha(G)=k+2$. Since $j\geq 1$, we have that $a$ is anticomplete to
$\{y_1^1, \dots, y_{d-1}^1\}$
and $w_1 \in W$. Let
$$S=\{a,w_1\} \cup \bigcup_{i \in \{1, \dots, k\};\  k-i \text{ odd}}\{y_1^i\} \cup
\bigcup_{i \in \{1, \dots, k\};\  k-i \text{ even}} \{z_1^i\}.$$
Then $S$  a stable set of size $k+2$ in $G$.

So far we have proved that $\alpha(X(d,k,j))=k+1$
for every $d \geq 2, k \geq 1$ and $d-1\geq j \geq 0$,
and  that $\alpha(Y(d,k,j))=k+2$
for every $d \geq 3, k \geq 1$ and $d-1\geq j \geq 1$.
The remaining equalities follow by a direct computation.
\end{proof}

Observe that $W_{d,1}$ is the $1$-skeleton of the $d$-dimensional
crosspolytope. Further,
%\begin{example}
%$W(3,2)$ is the $1$-skeleton of the icosahedron.
%\end{example}

\begin{observation}\label{flag2dim}
  The clique complex of $W_{3,k}$ is a flag $2$-sphere for every
  $k \geq 1$.
\end{observation}
\begin{proof}
For each $i$, $W_{3,k}[X_i]$ is a $4$-cycle.
Consider $W_{3,k}[X_i\cup X_{i+1}]$:
adding to the two disjoint $4$-cycles $W_{3,k}[X_i]\cup W_{3,k}[X_{i+1}]$ the edges
$y_s^iy_s^{i+1}$ and $z_s^iz_s^{i+1}$
(for the positive pairs $(y_s^iz_s^i,y_s^{i+1}z_s^{i+1})$ with $s=1,2$) makes a cylinder subdivided into $4$ squares; adding the other edges
for the positive pair with $s=1,t=2$ and for the negative pair with $s=2,t=1$
subdivides each of the four squares into two triangles.
Thus, $W_{3,k}[X_1\cup\ldots\cup X_k]$ is a triangulated cylinder, and adding $a,b$ with their edges makes a flag $2$-sphere.
\end{proof}

 Next we show:
\begin{theorem}\label{alpha2dim}
For every $n\ge 6$, $\alpha(3,n)=\lceil\frac{n}{4}\rceil$.
\end{theorem}
\begin{proof}
  Observe that $|V(X(3,k,j))|\equiv_4 2+j$,
  and $|V(Y(3,k,j))|\equiv_4 j$, and thus
  for every $n \geq 6$ there exist integers $k \geq 1$ and $j \geq 0$ and a
  graph $G \in \{X(3,k,j), Y(3,k,j)\}$ such that $|V(G)|=n$.
  Now by Theorem~\ref{alpha}
  for every $k \geq 1$ and $j \geq 0$ we have that
  $\alpha(X(3,k,j))= \lceil\frac{|V(X(3,k,j))|}{4}\rceil$,
  and for every $k \geq 1$ and $j \geq 1$ we have that
  $\alpha(Y(3,k,j))= \lceil\frac{|V(Y(3,k,j))|}{4}\rceil$.
 Finally, since $X''(3,k,j)$ and $Y''(3,k,j)$ are obtained from
 $cl(W_{3,k})$ by stellar edge subdivisions, it follows from
 Observation~\ref{flag2dim} that
 their clique complexes are flag $2$-spheres.
  We have shown that for every $n \geq 6$, $\alpha(3,n)\leq \lceil\frac{n}{4}\rceil$.
  Since by the 4CT every $n$-vertex triangulation of the $2$-dimensional sphere has a stable set of size $\lceil\frac{n}{4}\rceil$,
  $\alpha(3,n) \geq \lceil\frac{n}{4}\rceil$.
  \end{proof}

For $d=4$, the graph $W_{4,k}$ induces a cell structure on the $3$-sphere, consisting of tetrahedra with a vertex $a$ or $b$ and of triangular prisms consisting of a triangle on $X_i$ and the
corresponding triangle on $X_{i+1}$ (the corresponding vertices differ only in the
superscript). All these triangular prisms are triangulated by considering all tertrahedra defined by cliques of $W_{4,k}$ on this
set of 6 vertices, except for the following two (for a fixed $1\leq i\leq k-1$): $y_1^i,z_2^i,y_3^i; y_1^{i+1},z_2^{i+1},y_3^{i+1}$  and its ``antipodal prism"
$z_1^{i},y_2^{i},z_3^{i}; z_1^{i+1},y_2^{i+1},z_3^{i+1}$. We add the edge
$y_1^iz_2^{i+1}$ to triangulate the first, and the edge $z_1^iy_2^{i+1}$
to triangulate the second (such added edge is ``bent" inside the prism, the resulted triangulation of the prism is topological, not geometric); denote the resulting graph by $W'_{4,k}$.
Let $X'(4,k,j)$ and $Y'(4,k,j)$ be the graphs obtained from $X(4,k,j)$ and
$Y(4,k,j)$, respectively,  by adding the same edges.
See Figure~\ref{Fig:2} for an illustration of how the triangular prisms are triangulated.
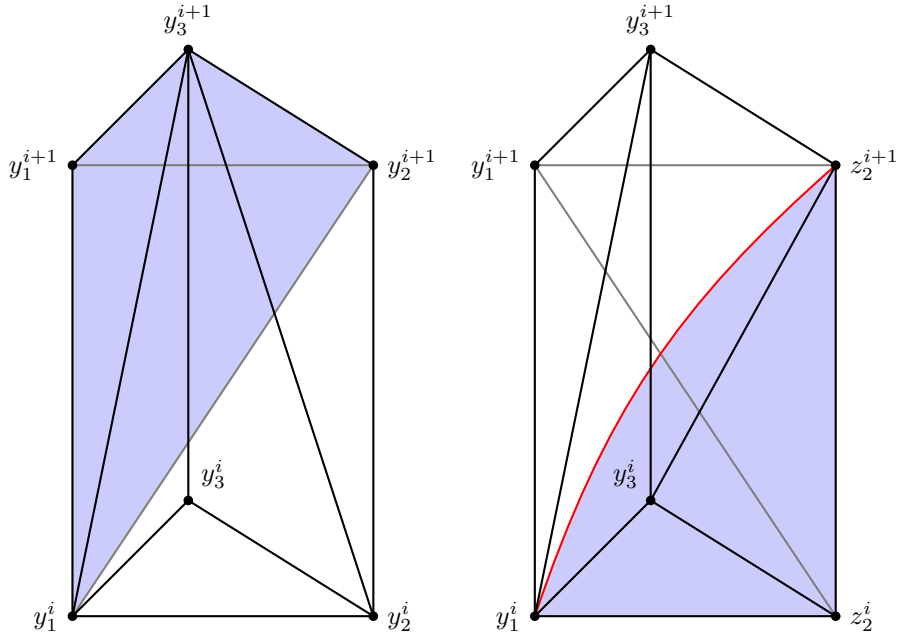
\begin{figure}[H]
\begin{minipage}{0.5\textwidth}
\begin{center}
	\begin{tikzpicture}%
	[
	scale=4,
	back/.style={gray,  thick},
	edge/.style={black, thick},
	vertex/.style={inner sep=1pt,circle,draw=black,fill=black,thick,anchor=base},
	facet/.style={fill=blue!95!black,fill opacity=0.2}
	]
	
%	%% draw coordinates
%	\coordinate (o) at (0, 0, 0);
%	\coordinate (x) at (1, 0, 0);
%	\coordinate (y) at (0, 1, 0);
%	\coordinate (z) at (0, 0, 1);
%	%% draw standard basis vectors
%	\node[vertex,label=right:{$x$}] at (x)     {};
%	\node[vertex,label=above:{$y$}] at (y)     {};
%	\node[vertex,label=below left:{$z$}] at (z)     {};
%	\draw[edge,red] (o) -- (x);
%	\draw[edge,green] (o) -- (y);
%	\draw[edge,blue] (o) -- (z);
	
	%% vertex locations
	\coordinate (t1) at (-1, 0.5, 0);
	\coordinate (t2) at (0, 0.5, 0);
	\coordinate (t3) at (-1, 0.5, -1);
		
	\coordinate (b1) at (-1, -1, 0);
	\coordinate (b2) at (0, -1, 0);
	\coordinate (b3) at (-1, -1, -1);
	
	%% facets before edges before vertices so they are "behind"
	\fill[facet] (b1) -- (t1) -- (t3) -- (t2) -- cycle {};
	%\fill[facet] (b1) -- (b3) -- (t3) -- (b2) -- cycle {};
	%\fill[facet] (b1) -- (b2) -- (t2) -- (t3) -- cycle {};
	
	%% back edges first
	\draw[edge,back] (t2) -- (b1);
	%% top triangle
	\draw[edge,back] (t1) -- (t2);
	\draw[edge] (t1) -- (t3);
	\draw[edge] (t2) -- (t3);
	
	%% bottom triangle
	\draw[edge] (b1) -- (b2);
	\draw[edge] (b1) -- (b3);
	\draw[edge] (b2) -- (b3);
		
	%% edges between top and bottom layer
	\draw[edge] (t1) -- (b1);
	\draw[edge] (t2) -- (b2);
	\draw[edge] (t3) -- (b3);
	
	\draw[edge] (t3) -- (b1);
	\draw[edge] (t3) -- (b2);

	%% draw vertices
	\node[vertex,label=left:{$y_1^{i+1}$}] at (t1)     {};
	\node[vertex,label=right:{$y_2^{i+1}$}] at (t2)     {};
	\node[vertex,label=above:{$y_3^{i+1}$}] at (t3)     {};
	
	\node[vertex,label=left:{$y_1^i$}] at (b1)     {};
	\node[vertex,label=right:{$y_2^i$}] at (b2)     {};
	\node[vertex,label=above right:{$y_3^i$}] at (b3)     {};
	
	\end{tikzpicture}
\end{center}
\end{minipage}
\begin{minipage}{0.5\textwidth}
\begin{center}
	\begin{tikzpicture}%
	[
	scale=4,
	back/.style={gray, thick},
	edge/.style={black, thick},
	vertex/.style={inner sep=1pt,circle,draw=black,fill=black,thick,anchor=base},
	facet/.style={fill=blue!95!black,fill opacity=0.2}
	]
	
	%	%% draw coordinates
	%	\coordinate (o) at (0, 0, 0);
	%	\coordinate (x) at (1, 0, 0);
	%	\coordinate (y) at (0, 1, 0);
	%	\coordinate (z) at (0, 0, 1);
	%	%% draw standard basis vectors
	%	\node[vertex,label=right:{$x$}] at (x)     {};
	%	\node[vertex,label=above:{$y$}] at (y)     {};
	%	\node[vertex,label=below left:{$z$}] at (z)     {};
	%	\draw[edge,red] (o) -- (x);
	%	\draw[edge,green] (o) -- (y);
	%	\draw[edge,blue] (o) -- (z);
	
	%% vertex locations
	\coordinate (t1) at (-1, 0.5, 0);
	\coordinate (t2) at (0, 0.5, 0);
	\coordinate (t3) at (-1, 0.5, -1);
	
	\coordinate (b1) at (-1, -1, 0);
	\coordinate (b2) at (0, -1, 0);
	\coordinate (b3) at (-1, -1, -1);
	
	%% facets before edges before vertices so they are "behind"
	\fill[facet] (b1) to (b2) to (t2) [bend right=15]to (b1);
		
	%% back edges first
	\draw[edge,back] (t1) -- (b2);
	\draw[edge, red, bend right=15] (t2) to (b1);
	%% top triangle
	\draw[edge,back] (t1) -- (t2);
	\draw[edge] (t1) -- (t3);
	\draw[edge] (t2) -- (t3);
	
	%% bottom triangle
	\draw[edge] (b1) -- (b2);
	\draw[edge] (b1) -- (b3);
	\draw[edge] (b2) -- (b3);
	
	%% edges between top and bottom layer
	\draw[edge] (t1) -- (b1);
	\draw[edge] (t2) -- (b2);
	\draw[edge] (t3) -- (b3);
	
	\draw[edge] (t3) -- (b1);
	\draw[edge] (t2) -- (b3);
			
	%% draw vertices
	\node[vertex,label=left:{$y_1^{i+1}$}] at (t1)     {};
	\node[vertex,label=right:{$z_2^{i+1}$}] at (t2)     {};
	\node[vertex,label=above:{$y_3^{i+1}$}] at (t3)     {};
	
	\node[vertex,label=left:{$y_1^i$}] at (b1)     {};
	\node[vertex,label=right:{$z_2^i$}] at (b2)     {};
	\node[vertex,label=above left:{$y_3^i$}] at (b3)     {};
	
	\end{tikzpicture}
\end{center}
\end{minipage}
\caption{Two triangular prisms with the induced graphs on their vertices. The grey edges indicate
edges not visible
from a front view of the depicted realization embeded in 3-space. The red edge is bent inside the right prism. In purple are sample induced tetrahedra. Note that in each prism, its clique complex triangulates it.}\label{Fig:2}
\end{figure}

%%%%%%%%%%%
\begin{theorem}\label{flag3dim}
  The clique complex of $W'_{4,k}$ is a flag $3$-sphere for every
  $k \geq 1$.
\end{theorem}
\begin{proof}
Recall the cell structure on the $3$-sphere described above, by tetrahedra and triangular prisms, induced by $W_{4,k}$.
First observe that for every triangular prism $T$ on vertex set $V(T)$ and every added edge $uv=y_1^iz_2^{i+1}$ or $z_1^iy_2^{i+1}$ of $W'_{4,k}$ on vertices in $V(T)$, all cliques in $W'_{4,k}$ involving $uv$ have their vertex sets contained in $V(T)$. Further, every clique of $W'_{4,k}[X_1\cup\ldots\cup X_k]$ has its vertex set contained in $V(T)$ some triangular prism $T$.
Hence, to show that $\cl(W'_{4,k})$ is a flag $3$-sphere it is enough to check that every induced subcomplex $\cl(W'_{4,k}[V(T)])$ triangulates the prism $T$.
Clearly the squares in each prism $T$ are triangulated, as exactly one diagonal in each square is inserted (which diagonal depends on whether the corresponding pair is positive or negative).
One verifies that each triangle on the boundary of $T$ is contained in exactly one tetrahedron whose vertex set is contained in $V(T)$, and there is no 5-clique whose vertex set is contained in $V(T)$. Thus, to verify that $\cl(W'_{4,k}[V(T)])$ triangulates the prism $T$ it suffices to
check for each tetrahedron $A$ whose vertex set is contained in $V(T)$ that each triangle $B$ in $A$ and not in the boundary of $T$, satisfies that $B$ is contained in exactly one more tetrahedron $A'$ whose vertex set is contained in $V(T)$. One inspects that this is indeed the case.
\end{proof}

%\begin{theorem}\label{edges}
%  For $W_{d,k}$ we have $f_0=2(d-1)k+2$ and $f_1=(2d-3)f_0-2d(d-2)$.
%\end{theorem}

Next we show:

\begin{theorem}~\label{alpha3dim}
  For all $n \geq 8$,
 $\alpha(4,n)\le \lceil\frac{n+1}{6}\rceil$.
\end{theorem}

\begin{proof}
  Observe that $|V(X(4,k,j))|\equiv_6 2+j$,
  and $|V(Y(4,k,j))|\equiv_6 j-1$ (here $0\le j\le 3$), and thus
  for every $n \geq 8$ there exist integers $k \geq 1$ and $j \geq 0$ and a
  graph $G \in \{X(4,k,j), Y(4,k,j)\}$ such that $|V(G)|=n$.
  Now by Theorem~\ref{alpha}
  for every $k \geq 1$ and $j \geq 0$ we have that
  $\alpha(X(4,k,j))= \lceil\frac{|V(X(4,k,j))|+1}{6}\rceil$,
  and for every $k \geq 1$ and $j \geq 1$ we have that
  $\alpha(Y(4,k,j))= \lceil\frac{|V(Y(4,k,j))+1|}{6}\rceil$.
 Since
$X'(4,k,j)$ and $Y'(4,k,j)$ are obtained from
$X(4,k,j)$ and $Y(4,k,j)$ by adding edges,
we deduce that
$\alpha(X'(4,k,j))\leq \lceil\frac{|V(X(4,k,j))|+1}{6}\rceil=\lceil\frac{|V(X'(4,k,j))|+1}{6}\rceil$
for every $k \geq 1$ and $j \geq 0$, and
$\alpha(Y'(4,k,j))\leq \lceil\frac{|V(Y(4,k,j))|+1}{6}\rceil=\lceil\frac{|V(Y'(4,k,j))|+1}{6}\rceil$
 and for every $k \geq 1$ and $j \geq 1$.

 Finally, since $cl(X'(4,k,j))$ and $cl(Y'(4,k,j))$ are obtained from
 $cl(W'(4,k))$ by stellar edge subdivisions, it follows from
 Theorem~\ref{flag3dim} that
 their clique complexes are flag $3$-spheres.
 This completes the proof.
  \end{proof}

\begin{remark}
In fact, $\alpha(X'(4,k,j))= \lceil\frac{|V(X(4,k,j))|+1}{6}\rceil$
for every $k \geq 1$ and $j \geq 0$, and
$\alpha(Y'(4,k,j))= \lceil\frac{|V(Y(4,k,j))|+1}{6}\rceil$
 for every $k \geq 1$ and $j \geq 1$.
\end{remark}
Indeed, for $d=4$ the sets $S$ constructed in the proof of Theorem~\ref{alpha}
are also independent in $X'(4,k,j)$ and $Y'(4,k,j)$ resp.

Finally we prove the lower bound of Theorem~\ref{tmh:alpha_bounds}.
\begin{theorem}\label{thm:alpha_lower}
Let $d\ge 4$. Then for all $n\ge 2d$,
$$ \alpha(d,n) \geq \frac{1}{4} n^{\frac{1}{d-2}}$$
\end{theorem}

\begin{proof}
  The proof is by induction on $d$. Let $\Delta$ be a $(d-1)$-flag sphere. Recall $\Delta$ has at least $2d$ vertices~\cite{Meshulam}, say it has $n$ vertices.

  For the base case let $d=4$. Then the link of $v$ in $\Delta$, denoted $\lk_v(\Delta)$, is a planar triangulation
  for every vertex $v$ of $\Delta$, and therefore, by  the 4CT,
  $\lk_v(\Delta)$ contains a stable set of size  $\lceil \frac{|V(\lk_v(\Delta))|}{4}\rceil$. Thus if for some vertex $v$ of $\Delta$ we
  have that $|V(\lk_v(\Delta))| \geq n^{\frac{1}{2}}$, then the theorem holds. If
  $|V(\lk_v(\Delta))| < n^{\frac{1}{2}}$ for every $v$, then a stable set
  of size $\frac{n}{ n^{\frac{1}{2}}}= n^{\frac{1}{2}}>\frac{1}{4} n^{\frac{1}{2}}$ can be obtained greedily.
  This finishes the case when $d=4$.

  Now we turn to general $d$.  In this case $\lk_v(\Delta)$ is a
  $(d-2)$-flag sphere   for every vertex   $v$ of $\Delta$,
  and therefore, inductively,
  $\lk_v(\Delta)$ contains a stable set of size
  $\frac{1}{4}|V(\lk_v(\Delta))|^\frac{1}{d-3}$.
    Thus if for some vertex $v$ of $\Delta$ we
  have that $|V(\lk_v(\Delta))| \geq n^{\frac{d-3}{d-2}}$, then the theorem holds. If
  $|V(\lk_v(\Delta))| < n^{\frac{d-3}{d-2}}$ for every $v$, then a stable set
  of size $\frac{n}{ n^{\frac{d-3}{d-2}}}= n^{\frac{1}{d-2}} > \frac{1}{4}n^{\frac{1}{d-2}}$ can be obtained greedily.
  This completes the proof.
  \end{proof}.

\iffalse
\begin{theorem} \label{C4}
  If $d>3$ the every edge of $W_{d,k}$ that is not incident with $a$ or
  $b$  in
  a $C_4$.
\end{theorem}

\begin{proof}
  UNFINISHED
  Clearly every edge with both ends in some $X_i$ is in a $C_4$ that is
  completely contained in $X_i$. By symmetry, it  remains to show that
  edges
  of the form $y_s^iy_t^{i+1}$, $y_s^iz_t^{i+1}$ and $z_s^iy_t^{i+1}$
  are in $C_4$'s.

  We consider the three cases separately.
\begin{enumerate}
\item {\bf An edge of the form  $y_s^iy_t^{i+1}$}. Then $t \geq s$.
    Suppose first that $s>1$.
  If there exists $l \in \{s, \dots, d-1\} \setminus \{t\}$, then
    $y_t^{i+1} \d  y_s^i \d z_{s-1}^i \d z_l^{i+1}-y_t^{i+1}$ is a
    $C_4$.
    It follows that $s=t=d-1$. But now  $y_{d-1}^{i+1} \d  y_{d-1}^i
    \d z_{d-2}^i \d y_{d-3}^{i+1}-y_{d-1}^{i+1}$ is a $C_4$. This
    proves that $s=1$.
\end{enumerate}
\end{proof}
\fi
%%%%%%%%%%%%%%%%%%%%%%%%%%%%%%%%%%%%%5
\section{Lower bounds on $f_1$}\label{sec:f1Turan}

The goal of this section is to prove Theorem~\ref{thm:flag-LBT-half}.
%
%\begin{theorem} \label{f1}
%  Let $d \geq 1$ be an integer and let $\Delta$ be a $(d-1)$-dimensional flag sphere. Then $f_1(\Delta) \geq \frac{\sqrt{2}+1}{2}df_0(\Delta)$.
%\end{theorem}

\begin{proof}
Let $\Delta=\cl(G)$ be a flag $(d-1)$-sphere on $n=f_0(\Delta)$ vertices and $f_1=f_1(\Delta)$ edges.
Let $\epsilon>0$, and assume $f_1<(d+\epsilon)n$.
We look for the largest $\epsilon=\epsilon(d)$ for which we reach a contradiction (when $d$ is chosen large enough, and then $n$ is chosen large enough w.r.t. $d$).

 % Suppose not, let $\Delta$ be a counterexample to \ref{f1}.
%  Write $f_0=f_0(\Delta)$ and $f_1=f_1(\Delta)$.
%  Then $f_1 < \frac{\sqrt{2}+1}{2} f_0$.

  By an easy  restatement of Tur\'an's theorem from \cite{AEKS}
  there is a stable set $I$ of $G$ with $|I| \geq \frac{n}{2(d+\epsilon)+1}$.

  We may assume $d\geq 4$. Then, we use the following well known facts: (i) $G$ is generically $d$-rigid, hence its space of stresses (a.k.a. affine $2$-stresses~\cite{Lee-stresses}) has dimension $g_2(\Delta):=f_1-dn+{{d+1} \choose {2}}$, see Kalai~\cite{Kalai:LBT}. (ii) For every vertex link, its graph is generically $(d-1)$-rigid and is not stacked
  (by flagness) hence, by the Cone Lemma, see e.g.~\cite[Cor.1.5]{TayWhiteWhiteley-skel1},
  for every vertex $v\in \Delta$ there exists a stress supported in the closed star of $v$ (namely in the induced graph of $G$ on $v$ and its neighbors) such that some edge containing $v$ has a nonzero weight.

   %  We will consider stresses in the face ring model, via the correspondence in Lee~\cite{Lee-stresses}.

  %The (second) cone lemma (see Adiprasito~\cite[Lem.3.3]{Adiprasito:toric},
  %and also the proof in Zheng~\cite[Prop.2.8]{Zheng:flag-survey}) shows that for every vertex $v\in G$, $g_1(\lk_v \Delta):=f_0(\lk_v (\Delta))-d$ equals the dimension of the degree $2$ part of the ideal $(x_v)$ in the ring  $\frac{\mathbb{R}[\Delta]}{(\Theta,\omega)}$, the face ring of $\Delta$ (over the reals) modulo a linear system of parameters $\Theta$ and
  %a Lefschetz element $\omega$ (yet another generic linear form).

  %Indeed, the following sequence with the obvious projection map is exact
  %$$ 0\rightarrow (x_v)\rightarrow \frac{\mathbb{R}[\Delta]}{(\Theta,\omega)}\rightarrow \frac{\mathbb{R}[\Delta \setminus v]}{(\Theta,\omega)}\rightarrow 0,$$
  %and $\dim_{\mathbb{R}}(x_v)_2=g_1(\lk_v\Delta)$ by the second cone lemma.

  %Also, by Lee~\cite{Lee-stresses}, $\dim_{\mathbb{R}}(\frac{\mathbb{R}[\Delta]}{(\Theta,\omega)})_2=g_2(\Delta)$.

  Now,
  %repeating an argument from the proof of~[Lem.4.2]\cite{Nevo-Missing},
  as $I$ is independent,
  %the sum of ideals $\sum_{v\in I}(x_v)$ in $\frac{\mathbb{R}[\Delta]}{(\Theta,\omega)}$ is direct, hence
the stresses mentioned above for $v\in I$ are linearly independent (each has a unique edge with a nonzero weight) and hence

  $$f_1 - dn+{{d+1} \choose {2}} \geq %\sum_{v\in I}(f_0(\lk_v \Delta)-d)
  %\geq |I|(d-2)
  |I|
  \geq \frac{n}
  {2(d+ \epsilon)+1},$$
  %\geq \frac{n (d-2)}
  %{2(d+ \epsilon)+1},$$
  %where the middle inequality follows as the octahedral sphere minimizes the number of vertices among all flag spheres of the same dimension (see e.g.~\cite[Thm.1.1]{Meshulam} or~\cite{Gal}).

Thus, for $n$ large enough w.r.t. $d$, we can ignore the ${{d+1} \choose {2}}$ term and get:  $\epsilon n > \frac{n}{2(d+\epsilon)+1}$,
%\frac{n (d-2)}{2(d+\epsilon)+1}$,
namely $\epsilon > \frac{1}{2(d+\epsilon)+1}$.
%$\epsilon > \frac{ (d-2)}{2(d+\epsilon)+1}$.

Solving the quadric for $\epsilon$ we get a contradiction if
$\epsilon <
\frac{-(2d+1)+\sqrt{(2d+1)^2 +8}}{4}$.
%\frac{-(2d+1)+\sqrt{(2d+3)^2-24}}{4}$.

Hence for arbitrarily small $\delta>0$, if $d$ is large enough we reach a contradiction for $\epsilon=\frac{1-\delta}{2d+1}$,
%$\epsilon=\frac{1}{2}-\delta$,
proving part (i).
For part (ii),
note that $\sqrt{x^2 +8}-x > \frac{3.95}{x}$ for $x\ge 13=2 \cdot 6 +1$,
%as $f(x)=\sqrt{x^2-24}-x$ is increasing over $x$ in the positive reals, then
thus for all $d\ge 6$ (and large enough $n$) we will reach a contradiction if
$\epsilon \le \frac{3.95}{4(2d+1)}=\frac{0.987}{2d+1}$.
%$\epsilon < \frac{-13+\sqrt{201}}{4}$,
%in particular for $\epsilon < 0.294$.
%
%
%
%  $$\frac{\sqrt{2}-1}{2}df_0 >  \frac{(d-1)(d-2)f_0}{2(\sqrt{2}+1)d}.$$
%
%  Therefore
%  $$d^2> d^2-3d+2,$$
%  a contradiction.
\end{proof}
Note that if Conjecture~\ref{conj:alpha} holds then plugging the larger value for $|I|$ yields
$f_1\ge (d+\frac{1}{2d-2})n$ for all $d\ge 6$ and large enough $n$.

\begin{conjecture}\label{conj:d+1-rigid}
For all $d\ge 5$, the graph of every flag $(d-1)$-sphere is $(d+1)$-rigid.
\end{conjecture}
If true, this conjecture would imply $f_1\ge (d+1)f_0-\binom{d+2}{2}$ for flag spheres of dimension $d-1 \ge 4$. A standard use of the Cone and Gluing Lemmas, see Kalai~\cite{Kalai:LBT}, reduces Conjecture~\ref{conj:d+1-rigid} to the case $d=5$. For $d<5$ its assertion is false.

\section{$\alpha_M(d,n)$}\label{sec:MaxStable}
Fix $d\ge 4$ and let $n \rightarrow \infty$. Then there exist simplicial $(d-1)$-spheres on $n$ vertices where the proportion of vertices in an independent set is arbitrarily close to $1$. To see this, start with the boundary complex $\Delta$ of a cyclic $d$-polytope with $m > d$ vertices, and note that $\Delta$ is a neighborly
$(d-1)$-sphere, i.e. all
$\binom{m}{\lfloor \frac{d}{2} \rfloor}$
subsets consisting of $\lfloor\frac{d}{2} \rfloor$ vertices are faces in $\Delta$. It is easy to check that $\Delta$ has
$\Theta(m^{\lfloor\frac{d}{2} \rfloor})$ facets. Perform stellar subdivisions on all facets. Then the set $I$ of the newly added
vertices is stable and of size $\Theta(m^{\lfloor\frac{d}{2} \rfloor})$ , while only the original $m$ vertices are not in $I$.

%neighborly $(d-1)$-sphere on $m$ vertices and stellar-subdivide all facets. The set $I$ of the new vertices is stable and of size $\Theta(m^{\lfloor d/2\rfloor})$ for large $m$, while only the old $m$ vertices are not in $I$.

In contrast, for flag spheres we conjecture that the proportion of vertices in an independent set can not exceed $1/2$.

\begin{conjecture}
For all $d\geq 2$, $\alpha_M(d,n)=\lfloor\frac{n-2(d-2)}{2}\rfloor$.
\end{conjecture}
%\eran{Can we prove it for all $d$?? Is it true that $1/2$ proportion for vertex links gives at most $1/2$ for the sphere, by average? Cauchy-Schwarts??}

This conjecture clearly holds for $d=2$ and we prove it for $d=3$.
The lower bound holds for all $d\ge 2$ by the following construction:
consider the $(d-2)$-fold suspension over the $(n-2(d-2))$-gon. A maximum stable set is obtained by taking every second  vertex along the $(n-2(d-2))$-gon.
\begin{theorem}\label{thm:maxstable2dim}
For all $n\geq 6$, $\alpha_M(3,n)=\lfloor\frac{n-2}{2}\rfloor$.
\end{theorem}
\begin{proof}
The construction above proves the lower bound $\alpha_M(3,n)\geq \lfloor\frac{n-2}{2}\rfloor$.
To show $\alpha_M(3,n)\leq \lfloor\frac{n-2}{2}\rfloor$, let $I$ be a maximum stable set in the graph $G=(V,E)$ of a flag $2$-sphere on $n$ vertices (it forces $n\ge 6$). Let $G'=(V,B)$ be the subgraph of $G$ whose edges are those with exactly one vertex in $I$. Then $G'$ is bipartite and planar. Further, $G'$ has at least two vertices in $I$ (as each vertex in $G$ has a non-neighbor) and at least two (in fact $4$) vertices in the complement of $I$ (as each vertex in $I$ has degree at least $4$ by flagness). Thus, $G'$ has at most $2n-4$ edges (this is known, see e.g.~\cite[Lemmas 4.2, 4.3]{Kalai-Nevo-Novik} for a proof).
On the other hand,
$$|B|=\sum_{v\in I}\deg(v) \geq 4|I|,$$
as each vertex in $G$ has degree at least $4$, and for all $v\in I$ the degree is preserved when passing to $G'$. Thus $4|I|\leq 2n-4$, hence $|I|\leq \lfloor\frac{n-2}{2}\rfloor$.
\end{proof}

%\textbf{Acknowledgements.}
%We thank Daniel Kalmanovich for producing the figures, and Hailun Zheng for helpful comments on the presentation.
%
%\bibliographystyle{plain}
%\bibliography{biblioStable}

%\acknowledgements{
\textbf{Acknowledgements.}
We deeply thank Isabella Novik and and Hailun Zheng for spotting a false statement in our ``proof" of a stronger version of Thm.1.3,
 Daniel Kalmanovich for producing the figures, and the anonymous referees of FPSAC2022 for helpful comments that greatly improved the presentation. An extended abstract of this work will be presented at FPSAC2022~\cite{FPSAC2022}.
 %}

%% if you use biblatex then this generates the bibliography
%% if you use some other method then remove this and do it your own way

%\printbibliography
\bibliographystyle{plain}
\bibliography{biblioStable}

%\begin{thebibliography}{99}
%
%\bibitem{AEKS} M. Ajtai, P. Erd\"{o}s, J. Koml\"{o}s, E. Szemeredi,
%  On Tur\'an's theorem for sparse graphs, {em Combinatorica} {\bf 1} (1981),
%  313--317.
%
%  \end{thebibliography}

\end{document}